\documentclass[11pt,a4paper,reqno]{amsart}
\pdfoutput=1
\usepackage[hidelinks,pdfstartview=FitH]{hyperref}
\usepackage{amssymb}
\usepackage[english]{babel}
\allowdisplaybreaks[4]

\renewcommand{\bar}[1]{\mkern 2mu\overline{\mkern-2mu#1}}

\newtheorem{thm}{Theorem}

\newtheorem{prop}[thm]{Proposition}

\newtheorem{ex}[thm]{Example}

\newcommand{\R}{\mathbb{R}}
\newcommand{\C}{\mathbb{C}}
\newcommand{\tr}{\mathrm{Tr}}

\newcommand{\inner}[1]{\left<\smash{#1}\right>}

\begin{document}
\title[A dual formula for the spectral distance in ncg]{A dual formula for the spectral distance\\[5pt] in noncommutative geometry}
\author{Francesco D'Andrea, Pierre Martinetti}

\address[F.~D'Andrea]{%Dipartimento di Matematica e Applicazioni,
Universit\`a di Napoli ``Federico II'' and I.N.F.N. Sezione di Napoli, Complesso MSA, Via Cintia, 80126 Napoli, Italy}
\email{francesco.dandrea@unina.it}
\address[P.~Martinetti]{%Dipartimento di Matematica e Applicazioni,
Universit\`a di Genova (DIMA)  and I.N.F.N. Sezione di Genova,
\linebreak  via
Dodecaneso 35, 16146 Genova, Italy}
\email{martinetti@dima.unige.it}



\begin{abstract}
In noncommutative geometry, Connes's spectral distance
is an extended metric on the state space of a $C^*$-algebra generalizing  Kantorovich's dual formula of the Wasserstein
  distance of order $1$ from optimal transport. It is
  expressed as a supremum. We present a dual
  formula -   as an infimum -  generalizing Beckmann's ``dual of the
  dual'' formulation of the Wasserstein distance.
 We then discuss some examples with matrix algebras, where
    such a  dual formula may be useful to obtain upper bounds for the distance.
\end{abstract}

\maketitle

%%% ======================================================================

\section{Introduction}

Noncommutative geometry  \cite{Connes:1994kx} extends Gelfand
  duality (that holds between locally compact Hausdorff spaces and
commutative C*-algebras) to Riemannian geometry.
In his reconstruction theorem \cite{connesreconstruct}, Connes has shown that
  all the geometric informations of an oriented  closed Riemannian (spin) manifold $M$ are
  encoded within a suitable triple of algebraic data, 
\begin{equation}
\label{eq:7bis}
A:=C^\infty(M), \quad H:=L^2(M, E),\quad D:=c\circ\nabla^E
\end{equation}
where $E\to M$  is an Hermitian vector
bundle equipped with a unitary action $c$ of the Clifford algebra
bundle and a connection $\nabla^E$ compatible with the Riemannian
metric, and $C^\infty(M)$ is the (commutative)
algebra of complex-valued smooth 
functions on $M$ that acts by pointwise multiplication on the Hilbert
space $L^2(M,E)$ of square integrable sections of the bundle.
Typically, if one is interested in the Riemannian structure
only, then $E$ is the bundle of complex differential forms on
$M$, with $D$ the Hodge-de Rham operator;  or --- if one is interested
also in the spin structure in case $M$ is a spin
manifold --- then $E$ is the bundle of spinors  with $D$ the associated Dirac operator (see e.g.~\cite{GVF}).

The triple \eqref{eq:7bis} is the canonical example of a \emph{spectral triple},
consisting in an associative --- but not necessarily commutative --- complex $*$-algebra $A$ acting faithfully on a Hilbert
  space $H$ through a bounded representation $\pi$, together with a selfadjoint operator $D$ on $H$ such that $[D,\pi(a)]$ is bounded
  and $\pi(a)(1+D^2)^{-1}$ is a compact operator for
all $a\in A$.

The reconstruction theorem \cite{connesreconstruct}  shows that if a spectral triple $(A,
H, D)$, with $A$ unital and commutative, satisfies some
additional assumptions first listed in \cite{Connes:1996fu}, then
there exists a  closed oriented Riemannian 
manifold $M$ such that $A=C^\infty(M)$. A spin manifold is rebuilt
 if one adds more assumptions.
Inspired by this geometric example, one refers to a spectral triple where $A$ is non-commutative as a \emph{noncommutative geometry}.

In physics, there are strong theoretical motivations
  to believe that at small scale, the geometry of spacetime is not
  well captured by the notion of manifold. Geometry is expected to
  become ``noncommutative'' (or ``quantum'', in a more physical
  language) at very small scale. There exists many attempts to  describe
  such ``noncommutative
  spaces'' in the mathematical physics
  literature (see e.g.~\cite{Martinetti:2009kx} and reference within). But Connes' approach is the only one that
  provides a definition of a distance that coincides with the
  geodesic distance in the commutative case, and still makes sense in the
  noncommutative framework.\footnote{Often, ``noncommutative spaces''
  are intended as objects whose ``coordinates do
  not commute''. In the most elaborated of these models, one is able to
  define a length operator. However,  the interpretation of the spectrum of this
  operator in terms of distances between some noncommutative equivalent of points is
  far from obvious \cite{Martinetti:2011fk}.}

Namely, given a spectral triple $(A, H, D)$, the  spectral distance between any
two states $\varphi_1, \varphi_2$ of $\bar{A}$ (the closure of $A$
with respect to the operator norm on $H$)
is~\cite{Connes:1992bc}:
\begin{align}\label{eq:oneA}
d_D(\varphi_1,\varphi_2) &:=\sup_{a\in A}\big\{|\varphi_1(a)-\varphi_2(a)|:\|[D,\pi(a)]\|\leq 1\big\}.
\end{align}
It is not difficult to check that, for an arbitrary spectral triple, the formula \eqref{eq:oneA} defines
a distance 
(more precisely an \emph{extended metric}, since \eqref{eq:oneA} may be infinite) on the space $S(\bar{A})$ of states of $\bar A$.
In the commutative case~\eqref{eq:7bis}, the spectral distance \eqref{eq:oneA} between pure states of the algebra $\bar A=C(M)$ of continuous functions on $M$, which are in $1$-to-$1$ correspondence with the points of $M$, coincides with the geodesic distance.\footnote{If $M$ is not connected, one adopts the convention that points in different connected components are at infinite geodesic distance.}
The meaning of the spectral distance between non-pure states, in the
commutative case, was first noticed by Rieffel in \cite{Rieffel:1999wq}. He pointed out that \eqref{eq:oneA} applied to \eqref{eq:7bis} is nothing but Kantorovich dual formulation of the Wasserstein distance $W_1$ of order $1$ studied in optimal transport, with cost function given by the geodesic distance $d_\text{geo}$ on $M$.
States of $C(M)$ -- that is  positive, normalized,  linear functionals $\varphi$ on $C(M)$ -- are in $1$-to-$1$ correspondence
with probability measures $\mu$ on $M$,
\begin{equation}
  \label{eq:9}
  \varphi(f) =\int_M f \, d\mu \quad \forall f\in C(M).
\end{equation}
One then has 
\begin{equation}
   \label{eq:10}
d_D(\varphi_1, \varphi_2) =W_1 (\mu_1, \mu_2)
\end{equation}
with
\begin{equation}
  \label{eq:8}
  W_1 (\mu_1,\mu_2) := \inf_{\gamma} \;\int_{M\times M}  \!\!
	d_\text{geo}(x,y)d\gamma(x,y)
	\;,
\end{equation}
where the infimum is on all probability measures $\gamma$ on $M\times M$ with marginals the measures $\mu_1$ ands $\mu_2$ defining the states $\varphi_1, \varphi_2$ through \eqref{eq:9}.
Later \cite{dAndrea:2009xr}, we show that \eqref{eq:10} actually holds even for non-compact 
complete manifolds, with $C^\infty(M)$ in \eqref{eq:oneA} 
replaced by the algebra of compactly supported smooth functions on $M$, or of smooth Lipschitz functions on $M$ vanishing at infinity.

\medskip

Connes' formula
\eqref{eq:oneA} thus provides a generalization of
Kantorovich dual formula  to the noncommutative
framework. In the commutative case, \eqref{eq:8} provides a
``pre-dual'' formula. A natural question is whether there 
exists a similar formula in the noncommutative case, that is a noncommutative version of \eqref{eq:8}. In other
terms, given an arbitrary spectral triple $(A, H, D)$, is there a way to express the spectral distance \eqref{eq:oneA} as an
infimum rather than a supremum?

The interest is not just theoretical. In most of the explicit
computations of the spectral distance, one usually ``guesses'' a good
upper bound, then shows this is saturated. The difficulty is
that while any $a\in A$ satisfying the side inequality $\|[D,\pi(a)]\|\leq 1$ gives a lower bound for the spectral distance:
\begin{equation}
d_D(\varphi_1,\varphi_2)\geq |\varphi_1(a)-\varphi_2(a)|,
\label{eq:5}
\end{equation}
there is no easy way to obtain an upper bound. On the contrary, in the commutative case, any probability measure $\gamma$ on $M\times M$ with marginals $\mu_1$ and $\mu_2$ gives an upper bound for the Wasserstein distance:
\begin{equation}\label{eq:5b}
W_1(\mu_1,\mu_2)\leq \int_{M\times M}  \!\!d_\text{geo}(x,y)d\gamma(x,y) .
\end{equation}
Notice that if upper and  lower bounds in \eqref{eq:5} and \eqref{eq:5b} coincide, the distance
is actually computed.
What is missing in the non-commutative case is a formula expressing $d_D$ as an infimum, from
which one could easily get upper bounds.
\medskip

A first attempt in this direction starts with the following observation: on a manifold, the cost function is retrieved as the Wasserstein distance between pure states. On an arbitrary
spectral triple, replacing states by probability measures on the
space $\mathcal{P}(\bar{A})$ of pure states, a natural candidate for a pre-dual formula
would be the Wasserstein distance $W_D$ on $\mathcal{S}(\bar A)$,
with cost the spectral distance on pure states (but note that the map
from probability measures on $\mathcal{P}(\bar A)$ to $\mathcal{S}(\bar{A})$ is surjective but, in general, not injective).
One shows \cite{Martinetti:2012fk} that $d_D(\varphi_1,\varphi_2) \leq W_D(\varphi_1, \varphi_2)$ for any states $\varphi_1,\varphi_2$, with equality when $\varphi_1$ and $\varphi_2$ are convex linear combinations of the same two pure states. However the equality does not hold in general 
(see the counterexample in \cite{Martinetti:2017}).

A similar approach is in \cite[Sec.~4.6]{D15}:
on the Berezin quantization $A$ of a compact homogeneous $G$-space $M$, with $G$ a connected compact semisimple Lie group, the symbol map associates to every quantum state a unique probability measure on $M$; one can then consider the Wasserstein distance between probability measures and show that it gives a distance on quantum states that is dual to a suitable seminorm on $A$, in the spirit of \eqref{eq:oneA}.

\medskip

In this note, we show that there exists a dual formulation of the
spectral distance \eqref{eq:oneA} by adapting to noncommutative
geometry the ``dual of the dual'' formula of the Wasserstein distance, also known as Beckmann's problem (cf.~\cite[Sec.~4.2]{San15}).
This is inspired by the recent work \cite{chen2017} on matrix algebras. Our main result is Theorem \ref{thm:2},
which shows that the ``dual of the dual'' formula holds in full generality.

In the following, we will omit the representation symbol and identify
an element of the algebra $A$ of a spectral triple with its
representation as a bounded operator on $H$. By $A^+$ we shall mean
the subalgebra of $\mathcal{B}(H)$ generated by $A$ and $1$.

\section{A dual formulation for the spectral distance}

Let $(A, H, D)$ be a spectral triple, and
\begin{equation}
\Omega^1_D(A):=\mathrm{Span}\big\{ a[D,b] :a,b\in A^+
\big\}\label{eq:13}
\end{equation}
the
$A$-bimodule of generalized $1$-forms. Denote by $\nabla$ the derivation:
\begin{equation}\label{eq:nabla}
\nabla:A\to\Omega^1_D(A)\;,\qquad\nabla a:=[D,a] \;.
\end{equation}
By definition of a spectral triple, $[D,a]$ is bounded for any $a\in
A$, so that $\Omega^1_D(A)$, and therefore $\mathrm{Im}(\nabla)$, is a subset of the algebra
$\mathcal{B}(H)$ of bounded operators on $H$.

Let $B$ be any (complex) Banach subspace of $\mathcal{B}(H)$ containing $\mathrm{Im}(\nabla)$, $B^*$ its Banach dual, i.e.~the set of
linear functionals $\Phi: B\to \C$ that are bounded for the operator norm
\begin{equation}
  \label{eq:14}
  ||\Phi||:= \sup_{b\in B:b\neq 0} \frac{|\Phi(b)|}{||b||} \;,
\end{equation}
and $\mathcal{L}(A,\C)$ the set of linear maps $A\to\C$.
We denote by
\begin{equation}
\nabla^*: B^*\to \mathcal{L}(A,\C)
\end{equation}
the pullback of $\nabla$:
\begin{equation}
\nabla^*\Phi(a):=\Phi(\nabla a) \;,\qquad\forall\;\Phi\in B^*,a\in A.
\end{equation}
Given a state $\varphi$ of $\bar{A}$ we will denote by $\varphi_0$ its restriction to $A$.

\begin{thm}\label{thm:2}
Let $(A,H,D)$ be a spectral triple and, for any $\varphi,\psi\in S(\bar{A})$ at finite distance (i.e.~$d_D(\varphi,\psi)<\infty$)
and any Banach space $B$ containing $\text{Im}(\nabla)$, define:
\begin{equation}\label{eq:two}
W(\varphi,\psi):=\inf_{\Phi\in B^*}\big\{\|\Phi\|:\nabla^*\Phi=\varphi_0-\psi_0\big\} \;.
\end{equation}
Then:
(i) the above inf is well-defined (the set of $\Phi$ satisfying the side condition is not empty),
(ii) it is in fact a minimum and 
(iii) one has
\begin{equation}\label{eq:1}
W(\varphi,\psi)=d_D(\varphi,\psi)
\end{equation}
In particular, \eqref{eq:two} is independent of the choice of $B$.
\end{thm}

\begin{proof}
Suppose $d_D(\varphi,\psi)<\infty$. Then $\varphi-\psi$ vanishes on $\ker\nabla$.
Indeed, let $b\in\ker\nabla$. Set $a:=\lambda b$, with $\lambda\in\C$. Since $[D,a]=0$, one has
$$
d_D(\varphi,\psi)\geq |(\varphi-\psi)(a)|=|\lambda|\cdot |(\varphi-\psi)(b)|\;\forall\;\lambda\in\C \;.
$$
If $(\varphi-\psi)(b)\neq 0$, the sup over all $\lambda$ is infinite and we get a contradiction.

Let $B_0:=\mathrm{Im}(\nabla)$ and $\Phi_0:B_0\to \C$ be the map given by:
\begin{equation}\label{eq:infmin}
\Phi_0(\nabla a):=\varphi_0(a)-\psi_0(a) \;,\qquad\forall\;a\in A.
\end{equation}
Since $\ker\nabla\subset\ker(\varphi_0-\psi_0)$, such a map is well defined.
Note that
\begin{equation}\label{eq:2}
\|\Phi_0\|:=\sup_{b\in B_0:b\neq 0}\frac{|\Phi_0(b)|}{\|b\|}=
\sup_{a\in A: [D,a]\neq 0}\frac{|\varphi(a)-\psi(a)|}{\|[D,a]\|}=
d_D(\varphi,\psi)
\end{equation}
is finite. By the Hahn-Banach theorem \cite[Pag.~77, Cor.~1]{RS80}, the map
$\Phi_0$ can be extended (non-uniquely) to a linear map in $B^*$ with
the same norm, which proves the first statement.

Let $\Phi\in B^*$ be any of these extensions. Since by construction
\begin{equation}
\nabla^*\Phi(a)=\Phi(\nabla a)=\Phi_0(\nabla a)=\varphi(a)-\psi(a)
\qquad \forall a\in A,
\end{equation}
we have from \eqref{eq:2}:
\begin{equation}
W(\varphi,\psi)\leq\|\Phi\|=\|\Phi_0\|=d_D(\varphi,\psi).
\label{eq:3}
\end{equation}
On the other hand, for all $\Phi\in B^*$,
\begin{equation}
\|\Phi\|=\sup_{b\in B:b\neq 0}\frac{|\Phi(b)|}{\|b\|}\geq
\sup_{b\in B_0:b\neq 0}\frac{|\Phi(b)|}{\|b\|}=
\sup_{a\in A:\nabla a\neq 0}\frac{|\nabla^*\Phi(a)|}{\|\nabla a\|}
 \;.
\end{equation}
In particular if $\nabla^*\Phi(a)=\varphi(a)-\psi(a)$, the right hand side of previous equation is $d_D(\varphi,\psi)$, which proves that
$W(\varphi,\psi)\geq d_D(\varphi,\psi)$.
The inf is attained on any extension $\Phi\in B^*$ of \eqref{eq:infmin} with the same norm, thus it is a minimum.
\end{proof}

\begin{prop}
Let $(A,H,D)$ be a spectral triple and $\varphi,\psi\in S(\bar{A})$ two states with $d_D(\varphi,\psi)=\infty$. Then, there is no $\Phi\in B^*$ satisfying $\nabla^*\Phi=\varphi_0-\psi_0$.
\end{prop}

\begin{proof}
Suppose $d_D(\varphi,\psi)=\infty$ and assume that there exists $\Phi\in B^*$ satisfying the side condition in \eqref{eq:two}. Then its restriction $\Phi_0$ to $\mathrm{Im}(\nabla)$ satisfies \eqref{eq:infmin} and from \eqref{eq:2} we deduce that $\Phi_0$ is unbounded, which is a contradiction since the norm of $\Phi_0$ should be bounded by the norm of $\Phi$.
\end{proof}

Assumming by convention that the inf \eqref{eq:two} on the empty set is $+\infty$, we have an equality between $W$ and $d_D$ for any pair of states.

\smallskip

It is worth noticing that, while \eqref{eq:two} is always a minimum, the supremum in \eqref{eq:oneA} is not always a maximum.

It is a maximum if the image of the set
$$
S:=\left\{a\in A^+, ||[D, a]||\leq 1\right\}
$$
under the quotient map $A^+\to A^+/\C 1$ is compact in the quotient topology induced by the norm topology on $A^+$.
This happens in particular if $A$ is finite dimensional and $\ker\nabla=\C1$: the latter condition implies that $L(a):= ||[D, a]||$ induces a norm on $A^+/\ker L=A^+/\C 1$, and $S/\C 1$ is the closed unit ball for this norm (which is compact since $A^+/\C 1$ is a finite-dimensional vector space). More generally, in the infinite-dimensional case and in the context of compact quantum metric spaces, the quotient unit ball is compact if $L$ is a \emph{closed Lip norm}, cf.~\cite[Theorem 5.2]{Rieffel:1999ec}. Since,
for all states $\varphi,\psi$, the difference $\varphi-\psi$ vanishes on constants as well, it defines a (continuous) function on $A^+/\C 1$. The spectral distance is the supremum of such a function on the closed unit ball $S/\C 1$, which is a maximum by Weierstrass theorem if the ball is compact.

As a corollary, if $S/\C 1$ is compact, then $d_D(\varphi,\psi)$ is finite for all $\varphi,\psi$ and the equality \eqref{eq:1} holds for any pair of states.

\section{Examples}
\subsection{Euclidean spaces}
In this section we recover Beckmann's
formula for the Wasserstein distance  in $\R^n$ from
  formula  \eqref{eq:two}. More precisely we show that the r.h.s. of the latter coincides with Beckmann's
  formula (given in \eqref{eq:BP} below).

Beckmann's
  formula deals with real valued functions, so first of all we need to check that the infimum in \eqref{eq:two} can be
  equivalently searched on real valued $\Phi$. This follows
  from the observation \cite{Iochum:2001fv} that
the supremum in the distance formula \eqref{eq:oneA} can be
equivalently searched on the set $ A^{\mathrm{s.a.}}$ of selfadjoint
elements of $A$, namely
\begin{equation}
d_D(\varphi_1,\varphi_2)= \sup_{a\in A^{\mathrm{s.a.}}}\big\{\varphi_1(a)-\varphi_2(a):\|[D,\pi(a)]\|\leq 1\big\}. 
\label{eq:oneB}
\end{equation}
Indeed, let $\nabla_{\R}$ be the restriction of
\eqref{eq:nabla} to the real vector space $A^{\mathrm{s.a.}}$. Take $B_{\R}$ to be any real
Banach subspace of $\mathcal{B}(H)$ containing
$\mathrm{Im}(\nabla_{\R})$, denote by $B^*_{\R}$ its dual and
$\nabla_{\R}^*:B^*_{\R}\to\mathcal{L}(A^{\mathrm{s.a.}},\R)$ the
pullback of $\nabla_\R$. Then, starting from \eqref{eq:oneB} and repeating almost
verbatim the proof of the above theorems, using the Hahn-Banach
theorem for real Banach spaces, one arrives at the following real version of Theorem \ref{thm:2}.
\begin{thm}
\label{thm:4}
Let $(A,H,D)$ be a spectral triple and 
$\varphi,\psi$ two states such that $d_D(\varphi,\psi)<\infty$. Then 
\begin{equation}\label{eq:dualbis}
d_D(\varphi,\psi)=\inf_{\Phi\in B_{\R}^*}\big\{\|\Phi\|:
    \nabla_{\R}^*\Phi=\varphi_0-\psi_0\big\}
\end{equation}
\end{thm}

We now specialize to the canonical spectral triple of 
$M=\R^n$. In such a case, $H=L^2(\R^n)\otimes\R^k$, where $k=2^{[n/2]}$ is the dimension of the spin representation, $D$ is the (flat) Dirac operator and $\nabla f=-i\sum_{\alpha=1}^n\gamma^\alpha\partial_\alpha f$ for all $f\in C_c^\infty(\R^n)$, with $\gamma^1,\ldots,\gamma^n$ Dirac's gamma matrices. Let $V\subset M_k(\C)$ be the real vector subspace of complex $k\times k$ matrices spanned by $i\gamma^1,\ldots,i\gamma^n$, equipped with Hilbert-Schmidt inner product:
\begin{equation}
\inner{a,b}_{\mathrm{HS}}:=\frac{1}k\tr(a^*b)\;,\qquad\forall\;a,b\in M_k(\C).
\end{equation}
With such a normalization, gamma matrices
$i\gamma^\alpha$ form an orthonormal basis of $V$:
\begin{equation}
\inner{i\gamma^\alpha, i\gamma^\beta}_{\mathrm{HS}}=\frac{1}{k}\tr(\gamma^\alpha\gamma^\beta)=\frac{1}{2k}\tr(\gamma^\alpha\gamma^\beta+\gamma^\beta\gamma^\alpha)=\frac{1}{k}\tr(\delta^{\alpha\beta}\mathbb I)=\delta^{\alpha\beta},
\end{equation}
with $\mathbb  I$ the identity matrix of dimension $k$.
The image of $\nabla_{\R}$ is contained in the Banach space
$B_{\R}:=C_0(\R^n,V)$ of continuous functions $\R^n\to V$ vanishing at
infinity, with norm inherited from the operator norm on $H$. Any $b\in
B_{\R}$ has the form $b=i\sum_{\alpha}f_\alpha\gamma^\alpha$, where
$\boldsymbol{f}:=(f_1,\ldots,f_n)$ is an $n$-tuple of real-valued
$C_0$-functions on $\R^n$. Since the operator norm is a $C^*$-norm
and 
\begin{equation}
b^*b=\sum_{\alpha,\beta}f_\alpha f_\beta\gamma^\alpha\gamma^\beta
=\frac{1}{2}\sum_{\alpha,\beta}f_\alpha
f_\beta(\gamma^\alpha\gamma^\beta+\gamma^\beta\gamma^\alpha)=\sum_\alpha
|f_\alpha|^2\mathbb I=|\boldsymbol{f}|^2 \mathbb I
\end{equation}
the norm on $B_{\R}$ is just the Euclidean norm in $\R^n$ composed with the sup norm:
\begin{equation}\label{eq:bnorm}
\|b\|=\sup_{x\in\R^n}|\boldsymbol{f}(x)|.
\end{equation}
By Riesz Representation Theorem (see e.g.~\cite[Theorem 6.19]{Rudin:1970fk}) the dual is $B_{\R}^*=\mathfrak{M}(\R^n)\otimes V^*$, with $\mathfrak{M}(\R^n)$ the set of (real-valued) Radon measures on $\R^n$. Every $\Phi\in B_{\R}^*$ can then be expressed in the form
\begin{equation}
\Phi(\,.\,)=i\sum_\alpha\int_{\R^n}\inner{\gamma^\alpha,\,.\,}_{\mathrm{HS}}d w_\alpha
\end{equation}
for some Radon measures $w_1,\ldots,w_n$. Given two states $\varphi,\psi$ corresponding to two measures $\mu,\nu$, the side condition $\nabla_{\R}^*\Phi=\varphi_0-\psi_0$ in \eqref{eq:dualbis} becomes:
\begin{equation}\label{eq:wcond}
\int_{\R^n} \boldsymbol{\nabla}f\cdot d\boldsymbol{w}=
\int_{\R^n} (f\,d\mu-f\,d\nu) \;,\qquad\forall\;f\in C^\infty_c(\R^n),
\end{equation}
where $d\boldsymbol{w}=(dw_1,\ldots,dw_n)$, $\boldsymbol{\nabla}=(\partial_1,\ldots,\partial_n)$ and ``$\cdot$'' is the standard Euclidean scalar product.
In particular, if $d\boldsymbol{w}=\boldsymbol{w}(x)d^nx$, $d\mu=\mu(x)d^nx$ and
$d\nu=\nu(x)d^nx$ for some $C^1$ functions, after integration by parts one gets:
\begin{equation}
\int_{\R^n} f\,\big\{\boldsymbol{\nabla}\cdot\boldsymbol{w}+\mu-\nu\big\}d^nx=0
\qquad\forall\;f\in C^\infty_c(\R^n).
\end{equation}
Since $C_c^\infty(\R^n)$
is dense in $C_0(\R^n)$,
the latter condition is equivalent to
\begin{equation}
\label{eq_w}
-\boldsymbol{\nabla}\cdot\boldsymbol{w}=\mu-\nu \;.
\end{equation}
For more general Radon measures, the divergence condition \eqref{eq_w} is to be read in the weak sense, i.e.~as in \eqref{eq:wcond}.

The norm on $B_{\R}^*$, dual to \eqref{eq:bnorm}, is
\begin{equation}
\label{eq:normphi}
    \|\Phi\|=\int_{\R^n} |\boldsymbol{w}(x)|d^nx .
\end{equation}
Combining \eqref{eq:normphi} and \eqref{eq_w}, after a replacement
$\boldsymbol{w}\to -\boldsymbol{w}$, Theorem~\ref{thm:2} yields 
\begin{equation}\label{eq:BP}
d_D(\varphi,\psi)=\min \left\{\,
\int_{\R^n} |\boldsymbol{w}(x)|d^nx  \;:\;\boldsymbol{\nabla}\cdot\boldsymbol{w}=\mu-\nu
\,\right\} \;.
\end{equation}
The r.h.s. of
the above equation is precisely Beckmann's formula (4.4) in
\cite{San15}. So we have recovered the later from the dual formulation of the
  spectral distance, as expected from Theorem \ref{thm:2}.

\begin{ex}
In simple cases, combining the formulas \eqref{eq:BP} and \eqref{eq:oneA} one is able to explicitly compute the distance. For $n=1$, for example, the difference between the cumulative distributions
\begin{equation}\label{eq:cum}
w(x):=\int_{-\infty}^xd(\mu-\nu)
\end{equation}
satisfies the side condition in \eqref{eq:BP}, hence we immediately get
\begin{equation}\label{eq:26b}
d_D(\varphi,\psi)\leq \int_{-\infty}^{+\infty}\left|\int_{-\infty}^xd(\mu-\nu)\right|dx \;.
\end{equation}
We know on the other hand that the sup in \eqref{eq:oneA} can be searched, in the commutative case, among the set of all $1$-Lipschitz functions (see e.g.~\cite{dAndrea:2009xr}). If $\mu-\nu$ has no singular continuous part, then \eqref{eq:cum} is piecewise continuous, its sign
is piecewise continuous and the primitive of the sign is a $1$-Lipschitz functions $\phi$. From \eqref{eq:oneA} it follows that:
\begin{equation}
d_D(\varphi,\psi)\geq\left|
\int_{-\infty}^{+\infty} \phi\,d(\mu-\nu)\right|=\int_{-\infty}^{+\infty}\left|\int_{-\infty}^xd(\mu-\nu)\right|dx
\end{equation}
so that \eqref{eq:26b} is an equality.
\end{ex}

\subsection{Finite noncommutative spaces}
In this section  we show how to recover the ``matricial distance'' of
\cite{chen2017} (more precisely their Theorem 1), which has been the
motivating example of the present work. To this aim, we consider the spectral triple
\begin{equation}\label{eq:29}
A=M_n(\C),\qquad
H=\C^n\otimes\C^N ,\qquad
D=\sum_{i=1}^NL_i\otimes E_{ii}
\end{equation}
where $n,N\geq 1$, the algebra $A$ acts on the first factor of $H$,
$L_1,\ldots,L_N\in M_n(\C)$ are selfadjoint matrices and $E_{ij}\in M_N(\C)$ is the matrix with $1$ in position $(i,j)$ and zero everywhere else.

Recall that states $\varphi$ are in bijection with density matrices $\rho\in M_n(\C)$, via the formula:
\begin{equation}
\varphi(a)=\tr(\rho a) \;,\qquad\forall\;a\in A.
\end{equation}
Choose the Banach space $B$ in \eqref{eq:two} as follows:
\begin{equation}
B=\underbrace{M_n(\C)\oplus\ldots\oplus M_n(\C)}_{N\text{ times}}
\end{equation}
where we think of elements of $B$ as block diagonal matrices with $N$ blocks of type $n\times n$.
The formula
\begin{equation}
\Phi=\sum_{i=1}^N\tr_{\C^n}(u_i\,\cdot\,)
\end{equation}
gives a bijection between $u=(u_1,\ldots,u_N)\in B$ and $\Phi\in B^*$. From
the cyclic property of the trace we deduce that
\begin{equation}
\nabla^*\Phi(a)=
\Phi([D,a])=
\sum_{i=1}^N\tr_{\C^n}(u_i[L_i,a])=
-\sum_{i=1}^N\tr_{\C^n}([L_i,u_i]a)
\end{equation}
for all $a\in A$, so that the side condition in \eqref{eq:two} becomes
\begin{equation}
-\sum\nolimits_{i=1}^N[L_i,u_i]=\rho_1-\rho_2
\end{equation}
for two states with density matrices $\rho_1$ and $\rho_2$.

The dual of the operator norm in the finite-dimensional case is the nuclear norm $\|\,.\,\|_*$:
\begin{equation}
\|\Phi\|=\|u\|_*:=\tr(\sqrt{u^*u}) \;.
\end{equation}
Collecting all these informations, and since $u$ and $-u$ have the
same norm, we arrive at the following theorem, specialization of our Theorem \ref{thm:2} to the present case.

\begin{prop}\label{thm:3}
The spectral distance of the triple \eqref{eq:29} is given by
\begin{equation}\label{eq:35}
d_D(\varphi_1,\varphi_2)=\inf\left\{ \|u\|_*:
\sum_{i=1}^N[L_i,u_i]=\rho_1-\rho_2
\right\} \;,
\end{equation}
where $\varphi_i=\tr(\rho_i\,\cdot\,)$, $i=1,2$, are any two states at finite distance and the infimum is over all
\mbox{$u=(u_1,\ldots,u_N)\in M_n(\C)\oplus\ldots\oplus M_n(\C)$} satisfying the side condition.
\end{prop}

Using \eqref{eq:dualbis} instead of \eqref{eq:two} one can show that the inf in \eqref{eq:35} can be searched in the set of antisymmetric matrices, thus getting the same formula that is in \cite{chen2017}.
\medskip

We close this section with a simple (well-known) example where the sup and inf formula combined allow to easily compute the spectral distance.

\begin{ex}
Consider the spectral triple \eqref{eq:29} with $n=2$, $N=1$ (hence $H=\C^2$) and $D=L_1:=\bigg(\!\begin{array}{rr}1 & 0 \\ 0 & -1\end{array}\!\bigg)$.
The correspondence
$$
\R^3\ni \vec x=(x_1,x_2,x_3)\mapsto \rho_{\vec x}:=\frac{1}{2}\bigg(\!\begin{array}{cc}1+x_3 & x_1-\mathrm{i}x_2 \\ x_1+\mathrm{i}x_2 & 1-x_3\end{array}\!\bigg)
$$
maps the closed unit ball into density matrices. We will denote by $\varphi_{\vec x}$ the state with density matrix $\rho_{\vec x}$.

Let $\vec x\neq\vec y$ be in the closed unit ball. Assume first that $x_3=y_3$. Define
$$
u=u_1:=\frac{1}{4}\bigg(\!\begin{array}{rc}0 & \overline{z} \\ \!-z & 0\end{array}\!\bigg)\;,\qquad
a:=\frac{1}{2|z|}\bigg(\!\begin{array}{cc}0 & \overline{z} \\ z & 0\end{array}\!\bigg)
$$
where $z:=x_1-y_1+\mathrm{i}(x_2-y_2)$ and $z\neq 0$ (since $\vec x\neq\vec y$). Both $u$ and $a$ satisfy the side conditions in \eqref{eq:oneA} and \eqref{eq:35}, so that
$$
 |\varphi_{\vec x}(a)-\varphi_{\vec y}(a)|\leq d_D(\varphi_{\vec x},\varphi_{\vec y})\leq \|u\|_* \;.
$$
Since lower and upper bound are both equal to $\frac{1}{2}|z|$, the distance is computed:
$$
d_D(\varphi_{\vec x},\varphi_{\vec y})=\frac{1}{2}\|\vec x-\vec y\| \;.
$$
If $x_3\neq y_3$, on the other hand, one easily checks that the distance is infinite.
\end{ex}

\end{document}